\documentclass[12pt, english]{article}          
\usepackage[utf8]{inputenc}
\usepackage{fancyhdr}
\usepackage{a4wide}
\usepackage{amsmath, amssymb, amsthm}            
\usepackage{amstext, amsfonts}
\usepackage{color}
\usepackage[colorlinks=true, linkcolor=red!70!black]{hyperref}
\hypersetup{citecolor=blue!70!black}
\usepackage{graphics,graphicx,subfigure}
\numberwithin{equation}{section}
\newtheorem{th1}{Theorem}[section]
\newtheorem{theorem}[th1]{Theorem}

\newtheorem{lemma}[th1]{Lemma}
\newtheorem{definition}[th1]{Definition}
\newtheorem{proposition}[th1]{Proposition}

\usepackage{tikz} 
\newcommand{\ve}{\varepsilon}
\newcommand{\vphi}{\varphi}

\newcommand{\as}{\backslash}
\usepackage{authblk}

\begin{document}	
	\title{\bf Estimates of Nonnegative Solutions to Semilinear Elliptic Equations\thanks{This work is supported by the Research Laboratory in Mathematics: Deterministic and Stochastic Modeling - LAMMDA.  {\it Code: LR16ES13.}}}
	\date{}
	\author[1]{Khalifa El Mabrouk\thanks{E-mail address: \texttt{khalifa.elmabrouk@essths.u-sousse.tn }}}
	\author[2]{Basma Nayli\thanks{Corresponding author: \texttt{basma.nayli@essths.u-sousse.tn}}}
	\affil[1,2]{Higher School of Sciences and Technology of Hammam Sousse,\newline University of Sousse, Tunisia}
	\maketitle
	\vspace*{-1cm}
	\begin{abstract}
	Let  $L$ be a second order uniformly elliptic  differential operator in a domain $D$ of~$\mathbb{R}^{d}$, $\psi:\mathbb{R}_+\to \mathbb{R}_+$ be a nondecreasing continuous function and let  $\xi,g:D\to\mathbb{R}_+$ be locally bounded Borel measurable functions. Under appropriate conditions, we determine a function $\vphi$ with values in $]0,1]$ such that  for every nonnegative solution to inequality $-Lu+\xi\psi(u) \geq  g$ in $D$ 	and for every $x\in D$,
	$$
	u(x)\geq p(x)\,\varphi\left(\frac{G_D(\xi\psi(p))(x)}{p(x)}\right)
	$$
	where $p=G_Dg$ is the Green function of $g$. The function $\varphi$ is completely determined by $\psi$ and  does not depend on $L,D,\xi$ or $g$.
	\end{abstract}

	
	\section{Introduction}
	
	Consider a domain $D$ of $\mathbb{R}^d,d\geq 1,$ and a second order uniformly elliptic  differential operator 
	$$
	Lu=\sum_{i,j=1}^{d}a_{ij}\partial_{ij}u+\sum_{i=1}^{d}b_i\partial_{i}u,
	$$
	where coefficients  $a_{ij}$ and $b_i$ are  sufficiently smooth. We are concerned with   nonnegative continuous solutions, in the distributional sense,  to  inequalities of type  
	\begin{equation}\label{nl-ieq}
		-Lu+\xi\psi(u) \geq g \quad\mbox{in }D,
	\end{equation} 	
	where  $g,\xi: D\rightarrow \mathbb{R}_+$ are locally bounded  Borel measurable  functions and    $\psi:\mathbb{R}_+\rightarrow \mathbb{R}_+$ is a nondecreasing continuous function such that $\psi(0)=0$. If~$D$ is bounded and regular, it is well known that, for every nonnegative continuous function $f$ on $\partial D$,  the   Poisson-Dirichlet problem 
		\begin{equation}\label{l-dir}
		\left\{\begin{array}{rcll}
			-Ls &=& g &\mbox{in }D,
			\\ s &=& f &\mbox{on }\partial D.
		\end{array}\right.
	\end{equation}  	
has a unique solution $s=S_D(f,g)$. Observe that $H_Df=S_D(f,0)$ is the $L-$harmonic extension of $f$ to $\overline{D}$ and  $G_Dg=S_D(0,g)$ is the Green potential of $g$.

 In this paper, it will first be shown that   for every nonnegative continuous  function~$f$ on  $\partial D$,		 the semilinear boundary value  problem
	\begin{equation}\label{psl}
		\left\{\begin{array}{rcll}
			-Lu+\xi\psi(u) &=& g &\mbox{in }D,
			\\u &=& f &\mbox{on }\partial D,
		\end{array}\right.
	\end{equation} 
	admits one and only one  solution which will be denoted by $U_D^\xi(f,g)$. Besides, it is simple to see that 
	\begin{equation}
		U_D^\xi(f,g)\leq S_D(f,g)\quad\mbox{in }D.
	\end{equation}
	Our main goal  is to establish a lower bound estimate of $U_D^\xi (f,g)$ as a function of~$S_D(f,g)$, which especially allows us to conclude lower bound estimates of any nonnegative solution to  inequality~\eqref{nl-ieq}.

	It should be noted that semilinear problems of type~\eqref{psl} play a vital role in the modeling of many phenomena in the fields of physics, astrophysics, logistics and differential geometry  (see \cite{marcus2014nonlinear,dynkin2002diffusions}). We also notice that~\eqref{psl} has been investigated in many papers, under various aspects, which deal with  existence and uniqueness as well as the estimation of its solution (see for instance \cite{ garcia2009existence, mohammed2010ground, fredj2013comparison,   Cao2016, Cao2017, hirata2018, ghardallou2019positive,  hirata2021} and references therein).

	Taking $\xi$ a positive constant in~$D$ ($\xi= 2$), Ben Fredj and the first author studied   problem~\eqref{psl} in~\cite{fredj2013comparison} when $L=\Delta$ is  the Laplacian operator  and $g$ is the null function. They proved  that for every nonnegative continuous function $f$ on $\partial D$,  \eqref{psl} has a unique solution~$u$ and  there exists a real constant $\kappa>0$ such that 	for all $x\in D$,
	\begin{equation}
	\kappa H_Df(x)\leq u(x) \leq H_Df(x).
	\end{equation}
	
	The case of $\xi\leq 0$  is recently investigated by  Grigor'yan and Verbitsky in~\cite{grigor2020}  where they  gave a real positive function $\varrho$, depending only on~$\psi$, such that if $S_D(f,g)\geq 1$ in $D$ then  every  nonnegative solution $u$ to problem~\eqref{psl} satisfies for all $x\in D$,
	\begin{equation}\label{urho}
		u(x)\geq \varrho\left(G_D\xi \right(x)).
	\end{equation}
	In this context,  M\^aagli and Zribi \cite{zribi2001} have already treated the special   setting where $L=\Delta$ and $f,g$ are identically zero.	

For  signed functions $\xi$, Grigor'yan and Verbitsky obtained in~\cite{grigor2019pointwise} 
	lower bounded estimates for nonnegative    solutions to 
	\begin{equation}\label{sup-sol-g}
		-Lu+\xi u^\gamma \geq  g \mbox{ in }D,
	\end{equation}  
	which are twice continuously differentiable  in $D$.

	In the current paper, we deal with nonnegative solutions  to inequality~\eqref{nl-ieq} in the distributional sense. We only consider functions $\xi$ that are nonnegative everywhere, however~$D$ can be any  Greenian domain of $\mathbb{R}^d$, not  necessary bounded, and $\psi$ belongs to  a  large class of  nondecreasing  functions  which contains  in addition to  polynomial functions, logarithmic functions as well as  exponential functions.
	
	We further  assume that $\psi$ is continuously differentiable in $ ]0,\infty[$ and there exists $c>0$ such that 
	\begin{equation}
		\psi(rt)\leq c\psi(r)\psi(t)\quad \mbox{for all }r\in[0,1], t\in\mathbb{R}_{+}, 
	\end{equation}
and consider the function   $\Theta:]0,1]\rightarrow [0, \ell[$ ($\ell$ may be infinite) given, for every $t\in ]0,1]$ by
	\begin{equation*}
\Theta(t)=\int_{t}^{1}\frac{1}{c\,\psi(s)}\,ds\quad\mbox{where}\quad	\ell =\int_{0}^{1}\frac{1}{c\,\psi(s)}\,ds.
\end{equation*}	
We obviously see that   $\Theta$ admits an inverse function which will be denoted by $\varphi$ and extended  to~$\mathbb{R}_+$ by setting $\varphi(t)=0$ if $t\geq \ell$. 
 Let $p=G_Dg$ and assume that   $0<p<\infty $ in~$D$. We will prove that every nonnegative solutions~$u$ to~\eqref{nl-ieq} satisfies, for all $x\in D$,
\begin{equation}\label{estine}
	u(x)\geq p(x)\varphi\left(\frac{G_D(\xi\psi(p))(x)}{p(x)}\right).	
\end{equation}
To do this, we will go through lower bounded estimates of solution $U_D^\xi(f,g)$ to the boundary value problem~\eqref{psl}. In this regard,  we had to carry out  a detailed study of $L$-harmonic functions in order to show, when $D$ is bounded regular, that for any  continuous function $f:\partial D\to\mathbb{R}_+$ such that $s=S_D(f,g)$ is finite not identically zero and $G_D(\xi\psi(s))$ is finite in~$D$, the integral equation 
\begin{equation}\label{eqint}
u+G_D\big(\xi\psi(u)\big)=s\quad\mbox{in }D,
\end{equation}
possesses a unique solution $u$ and that  for every $x\in D$,
\begin{equation}\label{estin}
s(x)\varphi\Big(\frac{G_D(\xi\psi(s))(x)}{s(x)}\Big)\leq u(x) \leq s(x).
\end{equation}
Notice that in the setting where Green potentials  $G_D\xi$ and $ G_Dg$  are continuous in $\overline{D}$ vanishing on $\partial D$, $u$ is a solution to~\eqref{eqint} if an only if $u$ is a solution to problem~\eqref{psl}.

For every $n\geq 1$,  we define  $v_n=U_{D_n}^\xi(0,g)$  where 
$(D_n)_{n\geq 1}$ is a sequence of  bounded regular domains such that for every $n\geq 1$,
$$
\overline{D_n}\subset D_{n+1} \quad\mbox{and}\quad \cup_{n\geq 1}D_n=D.
$$
Define $p_n=G_{D_n}g$ and choose  a subsequence $(p_{n_k})_{k\geq 1}$ of $(p_n)_{n\geq 1}$ such that $p_{n_k}>0$ for all~$k\geq 1$. We then apply~\eqref{estin} to $v_{n_k}$ in $D_{n_k}$, to obtain that for all $k\geq 1$,
\begin{equation}\label{estst}
v_{n_k}\geq p_{n_k} \vphi\left(\frac{G_{D_{n_k}}(\xi\psi(p_{n_k}))}{p_{n_k}}\right).  
\end{equation}
Given a solution $u$ to inequality~\eqref{nl-ieq}, it follows from a comparison principal that $u\geq v_{n_k}$ in $D_{n_k}$, which in turn allows us to get estimate~\eqref{estine}  in view of some  convergence theorems  and  monotonicities of functions~$\psi$ and $\varphi$.

Estimates of type~\eqref{estine} encompass  results given by Grigor'yan and Verbitsky  \cite{grigor2019pointwise}  for $\psi(t)=t^\gamma$ with $\gamma>0$. In fact, in this particular case,  simple computation of $\varphi$ yields   that every nonnegative solution $u$ to~\eqref{nl-ieq} satisfies, for all $x\in D$,
\begin{equation}
	u(x)\geq \left\{\begin{array}{ll}
		p(x)\left(1+(\gamma-1)\frac{G_D\left(p^\gamma \xi\right)(x)}{p(x)}\right)_+^{1/(1-\gamma)} & \mbox{if }\gamma<1,
		\\ p(x)\exp\left({-\frac{G_D(\xi p)(x)}{p(x)}}\right) & \mbox{if }\gamma=1, 
		\\ p(x)\left(1+(\gamma-1)\frac{G_D(p^\gamma \xi)(x)}{p(x)}\right)^{1/(1-\gamma)} & \mbox{if }\gamma>1.
	\end{array}\right.
\end{equation}	
 Our approach is still valid for various functions $\psi$ such as 
 $$
 \psi(t)=at+bt^\gamma,  \quad \psi(t)=\sinh t\quad \mbox{and} \quad \psi(t)=a(1+bt)\log(1+bt)
 $$
  where $a,b,\gamma>0$. These examples are presented and discussed in detail in the last section, where main results are stated and proved. In  section~2, we introduce some basic notions that are useful to  study the Dirichlet problem~\eqref{l-dir}. After treating fundamentals properties of the Green potential in section~3, we study in section~4 the semilinear problem~\eqref{psl} in regular domains. We prove  the existence and uniqueness of the nonnegative continuous solution to problem~\eqref{psl} and investigate its regularity. In section~5, we are  concerned with  estimates of the solution to~\eqref{psl} as well as estimates of nonnegative  solutions of the inequality~\eqref{nl-ieq}. 
	
	\section{$L-$harmonic functions}
	For every domain $D$ in $\mathbb{R}^d$,  $d\geq 1$, let $\mathcal{B}(D)$ be the set of all real-valued Borel measurable functions in~$D$ and let~$\mathcal{B}_{b}(D)$ (resp.~$\mathcal{B}^+(D)$) be the set of all  functions in~$\mathcal{B}(D)$ which are bounded (resp. nonnegative) in~$D$. More generally, given any class  of numerical functions~$\mathcal{F}$, we will denote by $\mathcal{F}_{b}$ (resp.~$\mathcal{F}^{+}$) the set of all functions in~$\mathcal{F}$ which are bounded (resp. nonnegative). We define~$\mathcal{C}(D)$ to be the set of all real-valued continuous functions in~$D$ and~$\mathcal{C}_{c}^{\infty}(D)$ to be the set of all infinity differentiable functions with compact support in $D$. For every $1\leq i,j\leq d$, $\partial_{i}$ stands for the first partial derivative with respect to the coordinate $x_{i}$ of $x=(x_{1},....,x_{d})\in \mathbb{R}^d$,  and $\partial_{ij}:=\partial_{i}\partial_{j}$.
	Let $\mathcal{C}^{1}(D)$ (resp. $\mathcal{C}^{2}(D)$) be the set of all functions $u\in\mathcal{C}(D)$ such that $\partial_{i}u\in\mathcal{C}(D)$  (resp. $\partial_{i}u, \partial_{ij}u\in\mathcal{C}(D)$) for every $1\leq i,j\leq d$. If $D$ is bounded, we denote by  $\mathcal{C}_{0}(D)$ the set of all functions $u\in \mathcal{C}(D)$  such that  $\lim_{x\in D, x\rightarrow z}u(x)=0$ for all $z\in \partial D$.
	
	Let $D$ be a domain in ${\mathbb{R}}^d$ and let   $0<\alpha \leq 1$. A real-valued function $u$ is called  $\alpha-$H\"{o}lder  continuous  in $\overline{D}$ if there exists a real constant $C>0$ such that 
	$$
	|u(x)-u(y)|\leq C |x-y|^\alpha\quad\mbox{for all }x,y\in\overline{D}.
	$$
	The set of all  $\alpha-$H\"{o}lder  continuous functions  in $\overline{D}$ is denoted by $\mathcal{C}^{\alpha}(\overline{D})$. We say that $u$ is locally $\alpha-$H\"{o}lder  continuous in the domain $D$ and we write $u\in\mathcal{C}^{\alpha}(D)$  if $u\in\mathcal{C}^{\alpha}(\overline{V})$ for all domains $V\Subset D$. Here and in the sequel, writing  $V\Subset D$ means that $\overline{V}$ is compact with closure $\overline{V}\subset D$.  We also define $\mathcal{C}^{1,\alpha}(D)$ (resp. $\mathcal{C}^{2,\alpha}(D)$) to be the class of all functions $u\in \mathcal{C}^{\alpha}(D)$ such that $\partial_{i}u\in \mathcal{C}^{\alpha}(D)$ (resp. $\partial_{i}u, \partial_{ij}u\in \mathcal{C}^{\alpha}(D)$) for all $1\leq i,j\leq d$. 
	
	We consider a second order  differential operator of type
	\begin{equation}
		Lu:=\sum\limits_{i,j=1}^{d}a_{ij}\partial_{ij}u+\sum\limits_{i=1}^{d}b_{i}\partial_{i}u
	\end{equation}
	where $a_{ij}\in\mathcal{C}^{2,\alpha}(\mathbb{R}^d)$ and  $b_{i}\in\mathcal{C}^{1,\alpha}(\mathbb{R}^d)$ for all $1\leq i,j\leq d$ and we assume that  $L$ is uniformly elliptic, that is, there exists a real  constant $\beta>0$ such that for all $x\in\mathbb{R}^d$ and for all $\xi=(\xi_1,\cdots,\xi_d)\in\mathbb{R}^d$,
	$$
	\sum_{i,j=1}^{d}a_{ij}(x)\xi_i\xi_j\geq \beta \,|\xi|^2.
	$$
	Given a Radon measure $\mu$ in $D$ and a locally Lebesgue integrable function $u$ in the open set $D$, equation
	\begin{equation}
		Lu=\mu \mbox{ in } D
	\end{equation} 
	should be understood in the distributional sense, that is,  for all $\varphi \in \mathcal{C}_{c}^{\infty}(D)$,
	\begin{equation}
		\int_{D}uL^{*}\varphi(x)\,dx = \int_{D}\varphi(x)\,d\mu(x) 
	\end{equation}
	where $L^{*}$ denotes  the adjoin operator of $L$ given by
	\begin{equation}
		L^{*}u:=\sum\limits_{i,j=1}^{d}\partial_{ij}(a_{ij}u)-\sum_{i=1}^{d}\partial_{i}(b_{i}u).
	\end{equation}
	We call $L-$harmonic in  $D$ any continuous function $u:D\to\mathbb{R}$ such that  $Lu=0$ in $D$ and define $\mathcal{H}_{L}(D)$ to be the set of all $L-$harmonic functions in $D$.
	We also assume that $L$ is hypoelliptic in the sense that $\mathcal{H}_{L}(D)\subset \mathcal{C}^2(D)$. 
	
	An open bounded set~$D\subset\mathbb{R}^d$  is called regular if for each continuous function $~{f:\partial D\rightarrow \mathbb{R}}$, the Dirichlet problem
	$$
	\left\{\begin{array}{rl}
		Lh= 0 &\mbox{ in } D,
		\\h=f &\mbox{ on } \partial D,
	\end{array}\right.
	$$
	has a unique solution  which will be denoted by $H_{D}f.$ In this setting, for every $x\in D$, the map $f\mapsto H_{D}f(x)$ defines a probability measure $H_{D}(x,\cdot)$ on $\partial D$ which is called harmonic measure relative to $x$ and $D$. For $x\in \mathbb{R}^d\as D$, it will be convenient to put $H_D(x,\cdot):=\delta_{x}$ where $\delta_{x}$ denotes the Dirac measure  concentrated at the point $x$. It is well known (see \cite[p.106]{gilbarg1998elliptic}) that $D$ is regular provided it satisfies the exterior sphere condition at every boundary point, that is, for every  $z\in\partial D$ there exists a ball $B$ of $\mathbb{R}^{d}$ such that $B\cap D=\emptyset$ and $z\in\partial B$. This immediately yields that every ball $B\subset \mathbb{R}^d$ is regular and consequently a continuous function $u$ is $L-$harmonic in $D$ if and only if $H_{B}u=u$ for all balls $B\Subset D$.
	
	Due to J. Serrin \cite{serrin1954harnack}, nonnegative $L-$harmonic functions in the domain $D$ satisfy  the following Harnack's inequality: For  every compact set $\Gamma\subset D$, there exists a real constant $C\geq 1$  (depending only on $L$, $\Gamma$ and $D$) such that for all $u\in\mathcal{H}_{L}^{+}(D)$,
	\begin{equation}\label{harnack}
		\sup_{x\in\Gamma}u(x)\leq C\inf_{x\in\Gamma}u(x).
	\end{equation}
	\begin{theorem}\label{bre_cv}
		Let $D\subset \mathbb{R}^d$ be a domain and let $(h_{n})_{n\geq 1}$ be a  nondecreasing  (or nonincreasing) sequence of  $L-$harmonic functions in $D$ such that $h:=\lim\limits_{n\rightarrow \infty}h_{n}$ is finite in some point $x_{0}\in D$. Then $h$ is $L-$harmonic in $D$.
	\end{theorem}
	\begin{proof}
		Let  $\Gamma$ be a compact subset of $D$ containing $x_0$ (otherwise, take $\Gamma'=\Gamma\cup \{x_{0}\}$) and let $\varepsilon>0$. It follows from the Harnack's inequality~\eqref{harnack} and the fact that $h_{n}(x_{0})$ converges to $h(x_{0})$ that, there exist  $n_{0}\geq 1$ and a real constant $C>0$ such that for every $x\in\Gamma$ and for every $m\geq n\geq n_0$,   we have
		$$
		0\leq h_{m}(x)-h_{n}(x)  \leq  C\big(h_{m}(x_{0})-h_{n}(x_{0})\big) \leq  C\big(h(x_{0})-h_{n}(x_{0})\big) \leq  C\varepsilon.
		$$
		Letting $m$ tend to $\infty$, we get that for every $x\in\Gamma$ and for every $n\geq n_0$,
		$$
		0\leq h(x)-h_{n}(x)  \leq  C\varepsilon.
		$$
		We then conclude that the sequence $(h_{n})_{n\geq 1}$ converges to $h$ locally uniformly in $D$. In particular,  $h$ is finite and continuous in $D$.  Now, for every  $B\Subset D$ and every $x\in B$,
		$$
		H_{B}h_{n}(x)=\int_{\partial B}h_{n}(y)H_{B}(x,dy)\underset{n\to\infty}{\longrightarrow} \int_{\partial B}h(y)H_{B}(x,dy)=H_{B}h(x).
		$$  
		Since $H_B h_n=h_{n}$ for every $n\geq 1$ and $h_n(x)$ converges to $h(x)$, we deduce that $H_Bh=h$ for every ball $B\Subset D$. This yields that $h$ is $L-$harmonic in $D$.
	\end{proof}
	\begin{proposition}\label{lem01}
		Let  $D\subset \mathbb{R}^d$ be a regular set  and let   $f:\partial D\to \overline{\mathbb{R}}_{+}$ be a Borel measurable function. Then, for every connected component $E$ of $D$,  $H_{D}f$ is either $L-$harmonic  or identically infinite   in $E$. 
	\end{proposition}
	\begin{proof} Without loss of generality, we assume that $D$ is connected. 
		In virtue of the previous theorem, we see that 
		$
		\mathcal{A}:=\{\Gamma\subset \partial D: H_{D}(\cdot, \Gamma)\in \mathcal{H}_{L}(D)\}
		$ 
		is a monotone class on $\partial D$. 
		Furthermore, $\mathcal{A}$ contains the set $\mathcal{O}$ of all open subsets of $\partial D$. In fact, let $O\in\mathcal{O}$ and consider 
		$$
		f_{n}:x\mapsto f_{n}(x)=\inf(1, n\,d(x, O^{c}))
		$$ 
		where $d(x,O^{c})=\inf_{y\in O^{c}}|x-y|$ denotes the distance between $x$ and  $O^{c}:=\mathbb{R}^d\as O$. Then $(f_{n})_{n\geq 1}$ is a nondecreasing sequence in $\mathcal{C}^{+}(\partial D)$ which converges to $\textbf{1}_{O}$ and hence $(H_Df_{n})_{n\geq 1}$ is a nondecreasing sequence of  $L-$harmonic functions in $D$ which converges to $H_D(\cdot,O)$ in view of the monotone convergence theorem.  We then conclude by Theorem~\ref{bre_cv},  that $H_{D}(\cdot, O)$ is $L-$harmonic and  thereby $\mathcal{O}\subset \mathcal{A}$. Therefore, it follows from the monotone class lemma that $\mathcal{A}$ contains $\mathcal{B}(\partial D)$. This yields that $H_{D}f\in \mathcal{H}_{L}(D)$ whenever 
		\begin{equation}\label{ef}
			f=\sum\limits_{i=1}^{k}\gamma_{i}\textbf{1}_{A_{i}}
		\end{equation} 
		with $k\geq 1$, $\gamma_{1},...,\gamma_{k}\in\mathbb{R}_+$ and $A_{1},...,A_{k}\in\mathcal{B}(\partial D)$. Finally, let $f:\partial D \to \overline{\mathbb{R}}_{+}$ be a Borel measurable function such that $H_{D}f\not\equiv \infty$ and choose a nondecreasing sequence $(f_{n})_{n\geq 1}$ of  nonnegative   functions of the form~\eqref{ef} which converges pointwise to  $f$. Then   $(H_{D}f_{n})_{n\geq 1}$ is a nondecreasing sequence of $L-$harmonic functions in $D$ pointwise converging to~$H_Df$  by the monotone convergence theorem. Thus,  $H_{D}f\in \mathcal{H}_{L}(D)$ in virtue of Theorem~\ref{bre_cv}.
	\end{proof}
	
	Let us point out that the previous result can be  extended to any Borel measurable function $f:\partial D\rightarrow \overline{\mathbb{R}}$ without having to assume that $f$ is   nonnegative in $\partial D$. More precisely, if  $D$ is connected and $f$ is $H_D(x_0,\cdot)-$integrable on $\partial D$ for some $x_0\in D$ then $f$ is $H_D(x,\cdot)-$integrable on $\partial D$ for all $x\in D$ and the function $H_Df$ is $L-$harmonic in~$D$.

	\medskip
	We recall that, due to the hypoellipticity of $L$  and \cite[Theorem 3.5]{gilbarg1998elliptic}, a non constant function $h\in~\mathcal{H}_L(D)$ can not achieve its  minimum (or maximum) in the interior of the domain $D$. This yields that every $L-$harmonic function $h$ in an open set $D$ is necessarily nonnegative provided $\liminf_{x\in D, x\to z}h(x)~\geq~0$ for every $z\in\partial D$.

	\medskip
	A lower semicontinuous function $v>-\infty$ is said to be $L-$superharmonic in $D$ if for every regular  set $U\Subset D$, $H_{U}v\in \mathcal{H}_{L}(U)$ and $ H_{U}v\leq v$ in~$U$.
	The set of all $L-$superharmonic functions in $D$ will be denoted by $\mathcal{S}_L(D)$. We will say that $v$ satisfies the property $(\mathcal M_*)$ if for every $x\in D$, there exists $r=r(x)> 0$ such that $B(x, r)\subset D$ and for all $0<\rho<r$, 
	\begin{equation}\label{promoy}
		\int_{\partial B(x,\rho)}v(z)H_{B(x,\rho)}(x,dz)\leq v(x).
	\end{equation} 
	It is obvious that every  $L-$superharmonic function satisfies property $(\mathcal M_*)$. The converse will be shown later.
	
	\begin{proposition}\label{prinmin}
		Let $D\subset\mathbb{R}^d$ be an open set and let $v:D\to]-\infty,\infty]$ be a lower semicontinuous function which satisfies the  property $(\mathcal M_*)$. The following holds true:
		\begin{enumerate}
			\item[a)] If $v$ attains its minimum in  a connected component $E$ of $D$, then $v$ is  constant in $E$.
			\item[b)] If $\liminf_{x\in D,\; x\to z} v(x)\geq 0$ for every $z\in \partial D$, then $v\geq 0$ in $D$.
		\end{enumerate}
	\end{proposition}
	\begin{proof} Without loss of generality, we consider the case where $D$ is connected and $v$ is not identically infinite in $D$. Assume that there exists $x_0\in D$ such that  $\min_{x\in D} v(x)= v(x_0)$. By~$(\mathcal M_*)$, there exists $r_{0}>0$ such that  $B(x_0,r_{0})\subset D$ and for every ball $B=B(x_0,\rho)$ with $0 <\rho < r_{0}$,
		$$
		H_Bv(x_0)=\int_{\partial B}v(z)H_B(x_0,dz)\leq v(x_0).
		$$
		Furthermore, $H_B v\geq v(x_0)$ in $B$ which yields that the $L-$harmonic function $H_B v$ reaches its minimum at the point $x_0$ and thereby $H_B v=  v(x_0)$ in $B$. 	
		On the other hand, since $v$ is lower semicontinuous and bounded below by  $v(x_0)$, there exists (see for instance \cite[Lemma 3.2.1]{armitage2001classical})  a nondecreasing sequence $(f_n)_{n\geq 1}$ of  continuous  functions in $\partial B$ such that $\lim_{n\to\infty}f_n(z)=v(z)$ for all $z\in\partial B$. Then $ H_B f_n\leq H_B v= v(x_0)$ in $B$ for every $n\geq 1$. Hence, for a given point  $z\in\partial B$, 
		\begin{equation}\label{eqlim}
			f_n(z)=\lim_{x\in B, \, x\rightarrow z} H_B f_n(x)\leq  v(x_0)\quad\mbox{for all }n\geq 1.
		\end{equation}
		The equality in the formula above follows from the fact that $f_n$ is continuous in $\partial B$ for all $n\geq 1$ and the ball $B$ is regular. 
		Letting $n$ tend to $\infty$, we obtain that $v(z)\leq  v(x_0)$ for every $z\in\partial B$ and so $v= v(x_0)$ in $\partial B(x_0,\rho)$. Since $0<\rho<r_0$ is arbitrary, we conclude that $v(x)= v(x_0)$ for all $x\in B(x_0,r_0)$.  Therefore, the set
		$$
		U=\{x\in D: v(x)= v(x_0)\}
		$$
		is open. By the lower semicontinuity  of~$v$, we also see that $U$ is  closed relative to $D$. Consequently $U= D$ and the proof of assertion $(a)$ is complete. In order to show $(b)$, let us assume that $\liminf_{x\in D, x\to z}v(x)\geq 0$ for every $z\in \partial D$ but there exists $x_0\in D$ such that $v(x_0)<0$. We extend $v$ to $\overline{D}$ by defining  
		$$
		v(z)=\liminf_{x\in D, x\to z} v(x)\mbox{ for all }  z\in \partial D.
		$$ It is clear that $v$ is lower semicontinuous in the compact set $\overline{D}$. Then, there exists $y_0\in \overline{D}$ such that   $\min_{x\in \overline{D}}v(x)=v(y_0)$. Since $v\geq 0$ in $\partial D$ and $v(x_0)<0$, we conclude that $y_0\in D$ and consequently $v$ achieves its minimum in the interior of~$D$. Using $(a)$ we get that $v(x)=v(x_0)<0$ for every  $x\in D$ which  yields a contradiction.  
	\end{proof}
	
	Throughout this paper, $\lambda_d$ will denote the Lebesgue measure in $\mathbb{R}^d$.	
	\begin{proposition}\label{mprop}
		Let $D\subset\mathbb{R}^d$ be an open set and let $v:D\to]-\infty,\infty]$ be a lower semicontinuous function which is finite $\lambda_d-$a.e in $D$. The following assertions are equivalent:
		\begin{enumerate}
			\item[a)] $v$ is $L-$superharmonic in $D$.
			\item[b)] $H_Bv\leq v$ for every ball $B\Subset D$.
			\item[c)] $v$ satisfies property $(\mathcal{M}_*)$ in $D$.
		\end{enumerate}
	\end{proposition}
	\begin{proof}
		We obviously see that $(a)$ yields $(b)$ which in turn implies $(c)$. To show that $(c)$ yields~$(a)$, we take a regular set $U\Subset D$ and choose a nondecreasing sequence $(f_n)_{n\geq 1}$ of continuous functions on $\partial U$ such that $\lim_{n\to\infty } f_n(z)=v(z)$ for all $z\in\partial U$. Then, for every  $n\geq 1$,  $w_n:=v-H_U f_n$
		is lower semicontinuous, $w_n>-\infty$ in $U$ and for all $z\in \partial U$,
		$$
		\liminf_{x\in U, x\to z} w_n(x)\geq v(z)-f_n(z)\geq 0.
		$$ 
		It is also immediate that $w_n$ satisfies~$(\mathcal M_*)$ in $U$. Thereby, according to assertion~$(b)$ in  the previous proposition, we conclude that $w_n\geq 0$ in $U$ for all $n\geq 1$ and consequently 
		\begin{equation}\label{fn}
			H_Uv=\lim_{n\to \infty} H_U f_n \leq v\quad \mbox{in }U.
		\end{equation}
		Furthermore, the $L-$harmonicity of $H_Uv$ follows from Proposition~\ref{lem01} and the fact that $v$ is finite  $\lambda_d-$a.e. Thus, $v$ is $L-$superharmonic in $D$ if $(\mathcal M_*)$ is satisfied.
	\end{proof}
	
	We end this section with the following convergence property of a monotone sequences of $L-$superharmonic functions which follows from Theorem~\ref{bre_cv}.
	\begin{proposition}\label{sconv}
		Let $D\subset \mathbb{R}^d$ be an open set and let $(v_{n})_{n\geq 1}$ be a  nondecreasing   sequence of nonnegative  $L-$superharmonic functions in $D$ such that $v:=\sup_{n\geq 1}v_{n}<\infty$  $\lambda_d-$a.e. Then $v$ is $L-$superharmonic in $D$.
	\end{proposition}

	\section{Green operator}
	We call Green function of $L$ in the regular set $D\subset\mathbb{R}^d$, the Borel measurable function 
	$G_{D}:D\times D\to ]0,\infty]$ having the following properties   for any  $y\in D$:
	\begin{enumerate}
		\item[G1)] $G_D(\cdot,y)$ is finite and continuous  in  $D\as\{y\}$ and  $\lim_{x\in D,x\to y}G_D(x,y)=G_D(y,y)$.
		\item[G2)]  $G_D(\cdot,y)$ is  	locally $\lambda_d-$integrable in $D$ and $LG_{D}(\cdot,y)=-\delta_{y}$.
		\item[G3)]  $\lim_{x\in D, x\rightarrow z}G_D(x,y)=0$ for all $z\in\partial D$.
	\end{enumerate}
	The fact that  coefficients $a_{ij},b_i$ are sufficiently smooth guarantees the existence of  the Green function  $G_D$ (see for instance \cite{miranda1970partial}).
	We also notice that a simple application of the maximum principle  in the domain $D\as\{y\}$ yields that  
	\begin{equation}\label{ypole}
		G_D(y,y)=\sup_{x\in D}G_D(x,y).
	\end{equation}

	\begin{proposition} Let $D\subset \mathbb{R}^d$ be a regular set, $y\in D$ and    $0<\beta< G_D(y,y)$. Then,  the  function $w:=G_D(\cdot,y)\wedge \beta$
		is $L-$superharmonic in $D$. In particular, $G_D(\cdot,y)$ is 	 $L-$superharmonic in $D$.		
	\end{proposition}
	\begin{proof} In virtue of Proposition~\ref{mprop}, it is sufficient to show that $w$ satisfies~$(\mathcal M_*)$. 	 Since  $LG_{D}(\cdot,y)=-\delta_{y}$, we  see that $w$ is $L-$harmonic in $\{G_D(\cdot,y)<\beta\}$ and consequently~\eqref{promoy} is valid for all $x\in D$ such that $G_D(x,y)<\beta$. If  $G_D(x,y)\geq \beta$,  inequality~\eqref{promoy} is immediate with $r(x)=d(x,D^c)$. Thus, $w=G_D(\cdot,y)\wedge\beta$ is $L-$superharmonic in $D$ for every   $0<\beta< G_D(y,y)$. To prove the last part of the proposition, we consider  $\beta_n=G_D(y,y)\wedge n$ and $w_n=G_D(\cdot, y)\wedge\beta_n$ for every $n\geq 1$. Therefore, Proposition~\ref{sconv} and~\eqref{ypole} yield that  $G_D(\cdot,y)=\lim_{n\to\infty}w_n$ is $L-$superharmonic in~$D$.
	\end{proof}
	
	We then conclude that $G_D(\cdot,y)$ is a {\em potential with support} $\{y\}$ in the sense of R.M. Hervé~\cite{herve1962recherches}. Hence,  the uniqueness of the Green function $G_D$ follows  from  Theorem~16.5 in~\cite{herve1962recherches} and property {\rm (G2)} above. In particular, if $D$ and $\Omega$ are regular sets such that  $D\subset\Omega$, then
	\begin{equation}\label{ggg}
		G_D(\cdot,y)=	G_\Omega(\cdot,y)-H_DG_\Omega(\cdot,y)\quad \mbox{in }D
	\end{equation}
	which yields that  $G_D(x,y)\leq G_\Omega(x,y)$ for all $x,y\in D$. 
	
	Given an arbitrary open set  $\Omega\subset\mathbb{R}^d$, we define for every $x,y\in \Omega$,
	\begin{equation}\label{greenf}
		G_{\Omega}(x,y):=\sup_{D\in \mathcal{R}(\Omega)} G_D(x,y),
	\end{equation}
where  $\mathcal{R}(\Omega)$ denotes the family of all regular sets~$D\Subset\Omega$.
	\begin{definition}
		The  open set $\Omega$ of $\mathbb{R}^d$ is called Greenian if for every $y\in \Omega$, $G_\Omega(\cdot,y)$ is not identically infinite in the connected component of $\Omega$ which contains $y$. In this setting, $G_\Omega$ is called Green function of $L$ in $\Omega$.
	\end{definition}
	
	If $\Omega $ is regular then $\Omega \in \mathcal{R}(\Omega)$ and $G_{\Omega}$ is just the Green function of $L$ in $\Omega$ given  at the beginning of this section. Let us also note that  for any sequence   $(D_n)_{n\geq 1}$  of regular sets such that $D_n\Subset D_{n+1}$ for all $n\geq 1$ and $\cup_{n\geq 1}D_n=\Omega$, we have
	\begin{equation}\label{lgreen}
		G_{\Omega}(x,y)=\lim_{n\to\infty } G_{D_n}(x,y).
	\end{equation}
	Hence, it is not difficult to show that $G_\Omega(\cdot,y)$ satisfies properties {\rm (G1)} and {\rm (G2)} for any  Greenian set~$\Omega$.  However, instead of {\rm (G3)}, we observe that $G_\Omega$ satisfies 
	\begin{enumerate}
	\item[G3')] For every $h\in\mathcal{H}_L^+(\Omega)$, if $h\leq G_\Omega(\cdot,y)$ in $\Omega$ then $h\equiv 0$ in $\Omega$.
	\end{enumerate}

	We define $\Gamma(x)=\overline\Gamma(|x|)$ for every $x\in\mathbb{R}^d$, where $\overline\Gamma$ is given for every real $r> 0$ by 	
	\begin{equation}\label{gacomp}
		\overline\Gamma(r)=\left\{\begin{array}{lr}
			r^{2-d} & \mbox{if }d\geq 3,
			\\ 1\vee(-\log r) & \mbox{if }d=2,
			\\ 1 & \mbox{if } d=1.
		\end{array}\right.
	\end{equation}
	Let $D\subset\mathbb{R}^d$ be a bounded open set. It is well known (see \cite{hueber1982}) that there exists a  real constant $c>0$  such that 
	\begin{equation}\label{equa}
		G_{D}(x,y)\leq c\, \Gamma(x-y)\quad\mbox{for all }x,y\in D.
	\end{equation}
	Therefore, every bounded open subset of $\mathbb{R}^d$, $d\ge 1$, is Greenian. Moreover, if $d\geq 3$ then~$\mathbb{R}^d$ and thereby all open subsets are Greenian. However, $\mathbb{R}$ and $\mathbb{R}^2$ are not Greenian.
	
	Given a Greenian set $D\subset\mathbb{R}^d$, the Green operator in $D$ is defined by 
	$$
	G_D p:=\int_D G_D(\cdot,y)p(y)dy
	$$
	for every function $p\in\mathcal{B}(D)$ for which the above integral exists. The following proposition is well proved in \cite{dynkin2002diffusions}.

	\begin{proposition}\label{pror}
		Let $D\subset\mathbb{R}^d$ be a bounded domain and let $p\in \mathcal{B}^+(D)$.
		\begin{enumerate}
			\item[a)]  If $p$ is bounded, then $G_Dp$ is bounded and  $\alpha-$H\"{o}lder  continuous in $D$. If moreover $D$ is regular, then $\lim_{x\in D, \, x\rightarrow z} G_Dp(x)=0$ for all $z\in \partial D$.
			\item[b)]  If $p$ is bounded and  $\beta-$H\"{o}lder  continuous in $D$ for $0<\beta\leq 1$, then $G_Dp\in C^2(D)$.
		\end{enumerate}
	\end{proposition}

	We need the following topological lemma which should be well known. For the completeness, we give here the proof. 
	\begin{lemma}\label{toplem}
		Let $D$ be a connected open subset of $\mathbb{R}^d$  and let $a\in D$. Then for every $x\in D$, there exist $n\geq 1$ and $x_0,x_1,\ldots, x_n\in D$ such that 
		\begin{equation}\label{segments}
			x_0=a,\;x_n=x\mbox{ and } [x_{k-1},x_k]\subset D\mbox{ for all }k\in\{1,\ldots,n\}.
		\end{equation} 
	\end{lemma}
	\begin{proof}
		Let $W$ be the set of all points of $D$ having the property given in the lemma. We claim that  $W$ is open. Indeed, let $x\in W$ and let $n\geq 1$, $x_0,x_1,\ldots, x_n\in D$ satisfying~\eqref{segments}. Choose $r>0$ such that  $B(x,r)\subset D$ and let $y\in B(x,r)$. 
		Then, setting $x_{n+1}=y$ we get $x_0,x_1,...,x_n, x_{n+1}\in D$ such that $x_0=a$, $x_{n+1}=y$ and $[x_{k-1},x_{k}]\subset D$ for every $1\leq k \leq n+1$. This means that  $B(x,r)\subset W$ and the claim is proved. To show that $W$ is a closed subset of~$D$, we consider a sequence  $(y_p)_{p\geq 1}$ in $W$ which converges to a point $y\in D$. Let $\varepsilon>0$ and $q\geq 1$ such that $y_q\in B(y,\varepsilon)$. Since $y_q\in W$, there exist $m\geq 1$ and $z_0,z_1,...,z_{m}\in D$ such that $z_0=a,z_m=y_q$ and $[z_{k-1},z_k]\subset D$ for~every $1\leq k\leq m$. To see that $y\in W$, it suffices to take $z_{m+1}=y$. Finally, we conclude that $W$ is open and closed in $D$ at the same time, which yields that $W=D$ since $W\not=\emptyset$ and $D$ is connected.
	\end{proof}
	\begin{theorem}\label{theoo} Let $D\subset \mathbb{R}^d$ be a bounded domain and let $p\in \mathcal{B}^+(D)$ be  a locally bounded function such that  $u:=G_Dp$ is finite at some point in~$D$. Then $u$ is  $\alpha-$Hölder continuous,  $L-$superharmonic in  $D$ and  $L-$harmonic in $D\as \overline{\{p>0\}}$.
	\end{theorem}
	\begin{proof}
		The proof is divided in three steps. Given a set $V\subset D$, we define $u_V=G_D(p\textbf{1}_V)$.
		
		{\em Step 1.}  We claim that if $u(x_0)<\infty$ for any  $x_0\in D$  then $u$ is finite and  continuous in the  ball~$B_r:=B(x_0,r)\Subset D$ where~$r>0$. Indeed, we decompose $u$ into the sum $u=v_{r}+w_{r}$ where $v_r=u_{B_r}$ and $w_r=u_{B_r^c}$. We obviously see that~$u_{B_r}$ is finite and  continuous in~$B_r$ because $p\textbf{1}_{B_r}$ is bounded in~$D$.
				On the other hand, given $0<\rho<r$ and $z\in B_r^c:=D\as B_r$, we easily see that $G_D(\cdot,z)$ is $L-$harmonic in~$D\as\{z\}$ and hence Fubini-Tonelli's theorem yields that for every $x\in B_\rho$,
		\begin{eqnarray*}
			H_{B_\rho}w_{r}(x)
			&=& \int_{D\as B_r}\left(\int_{\partial   B_\rho}G_{D}(y,z)\,H_{B_\rho}(x,dy)\right) p(z)\, dz
			\\ &=& \int_{D\as B_r}H_{B_\rho}G_{D}(\cdot,z)(x)\, p(z)\, dz
			\\ &=& \int_{D\as B_r}G_{D}(x,z)p(z)\,dz
			\\ &=& w_{r}(x).  
		\end{eqnarray*}
		Since $w_{r}(x_0)<\infty$, it follows that $H_{B_\rho}w_{r}\not\equiv\infty$ in $B_\rho$.  Then, in virtue of Proposition~\ref{lem01},  $w_{r}$ is $L-$harmonic in $B_\rho$ for every $0<\rho<r$. Whence  $w_{r}$ is $L-$harmonic in $B_r$, which yields in particular that $w_r$ is finite and   continuous in $B_r$. Consequently, $u$ is finite and  continuous in $B_r$.

		{\em Step 2.} Assume that $u(x)<\infty$ for some point $x\in D$ and let $y\in D$ such that  the segment $[x,y]\subset D$.  We claim that $u(z)<\infty$ for all  $z\in [x,y]$. Indeed, it is enough  to show that $u(y)<\infty$. To do this, put $\eta:=d([x,y], D^c)$ and let $0<r<\eta$. Choose an integer~$n$ such that   $nr>|y-x|$ and  define 
		$$
		z_k=x+\frac{k}{n}(y-x)  \quad\mbox{for }k=0,1,\ldots, n.
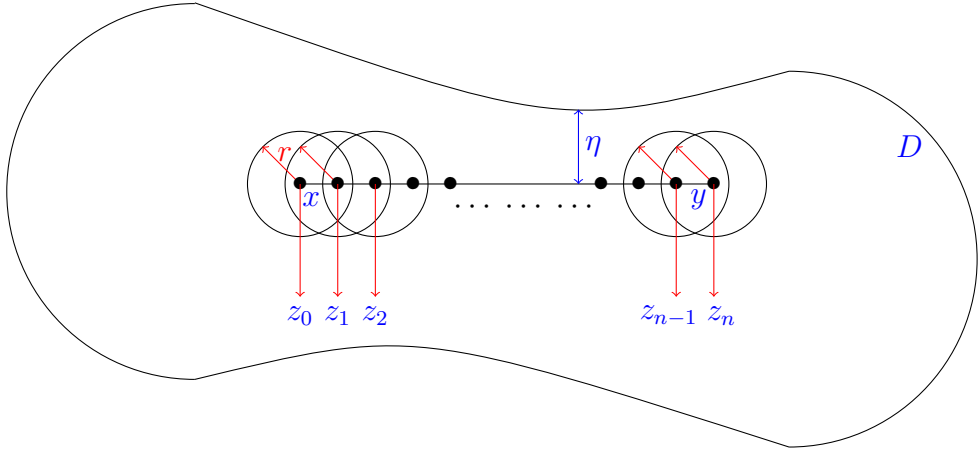
		$$
		\begin{center} 
			\begin{tikzpicture}
				\draw (2.1,5.91) .. controls (7,4.2)  .. (10,5) ;
				\draw (2.1,0.9) .. controls (5,1.6)  .. (10,0) ;
				\draw (3.5,3.5) -- (9,3.5);
				\draw (10,0) arc (-90:90:2.5cm);
				\node[blue]at (11.6, 4){\textbf{$D$}};
				\draw (2.1,0.9) arc (-90:-270:2.5cm);
				\draw (3.5,3.5) circle (0.7cm); 	
				\draw (4,3.5) circle (0.7cm);
				\draw (4.5,3.5) circle (0.7cm);
				\draw (8.5,3.5) circle (0.7cm);
				\draw (9,3.5) circle (0.7cm);
				\draw[->, red] (3.5,3.5)--(3,4); \node[red]at (3.3,3.9){\textbf{$r$}};
				\draw[->, red] (4,3.5)--(3.5,4);
				\draw[->, red] (8.5,3.5)--(8,4);
				\draw[->, red] (9,3.5)--(8.5,4);
				\foreach \Point in {(3.5,3.5),(4,3.5),(4.5,3.5),(5,3.5),(5.5,3.5),(7.5,3.5),(8,3.5),(8.5,3.5),(9,3.5)}{
					\node at \Point {$\bullet$};
				}
				\draw[<->, blue] (7.2,3.5)--(7.2,4.49);
				\node [black] at (6.5,3.2){$\cdots\;\cdots\;\cdots$};
				\node[blue]at (7.4,4){\textbf{$\eta$}};
				\node[blue]at (3.65,3.3){\textbf{$x$}};
				\node[blue]at (8.8,3.3){\textbf{$y$}};
				\draw[->, red] (3.5,3.5)--(3.5,2);
				\draw[->, red] (4,3.5)--(4,2);
				\draw[->, red] (4.5,3.5)--(4.5,2);
				\draw[->, red] (8.5,3.5)--(8.5,2);
				\draw[->, red] (9,3.5)--(9,2);
				\node[blue]at (3.5,1.75){\textbf{$z_0$}};
				\node[blue]at (4,1.75){\textbf{$z_1$}};
				\node[blue]at (4.5,1.75){\textbf{$z_2$}};
				\node[blue]at (8.4,1.75){\textbf{$z_{n-1}$}};
				\node[blue]at (9.1,1.75){\textbf{$z_n$}};
			\end{tikzpicture}
		\end{center}
		Seeing that $u(z_0)=u(x)<\infty$, we obtain by Step~1 that $u<\infty$ in ${B(z_0,r)}$ and then  $u(z_1)<\infty$ since  $|z_1-z_0| < r$. Using same arguments,  we deduce that $u(z_k)<\infty$ for every $k\in\{1,\ldots,n\}$. In particular  $u(y)=u(z_n)<\infty$ and the claim is proved.
		
		{\em Step 3.} From Steps~1 and 2, we first   conclude  that $u$ is finite and  continuous in the whole domain~$D$. To see this, let~$x_0\in D$ such that $u(x_0)<\infty$.  According to Lemma~\ref{toplem} above, for each $x\in D$  we can find $n\geq 1$ and $x_1,\ldots, x_n\in D$ such that $x_n=x$ and  $[x_{k-1},x_k]\subset D$ for all $1\leq k\leq n$.  Hence, by successive applications of Step~2, we get that $u(x)=u(x_n)<\infty$ and consequently~$u$ is finite in~$D$. The fact that $u$ is continuous in~$D$ follows from Step~1. Now, in order to prove that $u$ is $\alpha-$Hölder continuous in~$D$,  consider a domain  $V\Subset D$ and choose a domain~$W$ such that $V\Subset W\Subset D$. Since $p\textbf{1}_W$ is bounded in~$D$, it is clear that $u_W\in\mathcal{C}^\alpha(D)\subset \mathcal{C}^\alpha(\overline{V})$.
		 On the other hand, the function $u_{W^c}$ is continuous in~$W$ and same arguments as in Step~1 yield that $H_Bu_{W^c}=u_{W^c}$ for every ball $B\Subset W$. Therefore, $u_{W^c}$ is $L-$harmonic in~$W$ and consequently $\alpha-$Hölder continuous in~$W$. Thus $u=u_{W}+u_{W^c}\in\mathcal{C}^\alpha(\overline{V})$ for every domain $V\Subset D$ which means that $u\in\mathcal{C}^\alpha(D)$. 
	 It remains to show  that~$u$ is  $L-$superharmonic in~$D$  and $L-$harmonic in~$U:=D\as\overline{\{p>0\}}$. The first claim follows easily from  Fubini-Tonelli's theorem by observing that $G_D(\cdot,y)$ is $L-$superharmonic in~$D$ for all~$y\in D$. Finally, the function $u$ is $L-$harmonic in~$U$ because it is continuous and $H_Bu=u$ for any ball $B\Subset U$ by analogous calculus as in Step~1. 
	\end{proof}
	
We deduce that for every bounded domain $D$ of $\mathbb{R}^d$  and for every locally bounded  function $p\in \mathcal{B}^+(D)$ such that $G_Dp\not\equiv \infty$ in $D$, we have  
	\begin{equation}\label{gdp}
		LG_{D}p=-p\mbox{ in } D.
	\end{equation}
	Indeed, the previous theorem guarantees  that $G_Dp$ is finite and continuous in $D$ and  for every function  $\varphi\in\mathcal{C}^{\infty}_{c}(D)$, Fubini's theorem  yields   that
	$$
	\int_{D}G_Dp(x) L^{*}\varphi(x) dx=\int_{D}p(y)\Big(\int_{D}G_{D}(x,y)L^*\varphi(x)dx\Big)dy=-\int_{D}p(y)\varphi(y)dy.
	$$

	\begin{proposition}\label{pro2}
		Let $D\subset \mathbb{R}^d$ be a bounded domain. Then $G_{D}$ is a compact operator on~$\mathcal{B}_{b}(D)$ endowed with the uniform norm $\|\cdot\|_{\infty}$. 
	\end{proposition}
	\begin{proof}
		We have to show that $
		{\mathcal{F}}=\{G_{D}u: u\in\mathcal{B}_b(D)\mbox{ and }\|u\|_{\infty}\leq 1\}
		$
		is relatively compact in $\mathcal{B}_{b}(D)$ endowed with  $\|\cdot\|_{\infty}$. We obviously see that
		$
		\|G_{D}u\|_{\infty}\leq \|G_D1\|_{\infty} <\infty
		$ 
		for all $u\in\mathcal{B}_b(D)$ such that $\|u\|_{\infty}\leq 1$,		which means that  ${\mathcal{F}}$ is uniformly bounded. Thus, in view of  Ascoli's theorem, the proof will be finished provided we show that the family ${\mathcal{F}}$  is  equicontinuous in~$D$. In order to do this,  let $A\subset D$ be Borel set,  $x\in D$ and $r>0$. Then, in virtue of~\eqref{equa}, we have
		$$
		\int_{A}G_{D}(x,y)dy    \leq c\int_{B(x,r)}\Gamma(x-y)\,dy+c\int_{A\as B(x,r)}\Gamma(x-y)dy.
		$$ 
		By a simple computation, there exists a real constant  $c_1>0$,  depending only on $d$, such that   
		$$
		\int_{B(x,r)}\Gamma(x-y)\,dy\leq  c_1 \int_0^r \overline\Gamma(t)t^{d-1}\,dt \quad\mbox{and}\quad 
		\int_{A\as B(x,r)}\Gamma(x-y)dy\leq \overline\Gamma(r)\lambda_d(A).
		$$
		Therefore, for every $r>0$ we get that 
		$$
		\limsup_{\lambda_d(A)\to 0}\left(\sup_{x\in D}\int_{A}G_{D}(x,y)dy \right)\leq c\, c_1 \int_0^r \overline\Gamma(t)t^{d-1}\,dt. 
		$$
		Seeing that  $\int_0^r \overline\Gamma(t)t^{d-1}\,dt\to 0$ if $r\to 0$, we deduce that 
		$$
		\sup_{x\in D}\int_{A}G_{D}(x,y)dy \to 0\quad\mbox{if}\quad \lambda_d(A)\to 0.
		$$
		Whence  the family ${\mathcal{R}}=\{G_{D}(x,\cdot):x\in D\}$ is uniformly $\lambda_d$-integrable in $D$. Consequently, by  Vitali's convergence theorem, we conclude   that for all $x,z\in D$,
		$$
		\sup_{ \|u\|_{\infty}\leq 1}\left| G_Du (z)-G_Du(x)\right|\leq \int_{D}\left| G_{D}(z,y)-G_{D}(x,y)\right| dy \to 0 \quad\mbox{if}\quad z\to x.
		$$
		Thus 
		$\mathcal{F}$ is equicontinuous in the domain~$D$.    	
	\end{proof}
	\section{Semilinear Dirichlet problem }

	Throughout this work, we consider   a  regular domain~$D$ in $\mathbb{R}^d$, a Borel measurable locally bounded function $\xi: D\rightarrow \mathbb{R}_{+}$   and a nondecreasing continuous function  $\psi: \mathbb{R}_+\rightarrow \mathbb{R}_{+}$  such that $\psi(t)=0$ if and only if  $t=0$. We are concerned with existence and uniqueness of nonnegative solution to  the semilinear Dirichlet problem 
	\begin{equation}\label{base}
		\left\{\begin{array}{rcll}
			-Lu+\xi\psi(u) &=& g &\mbox{in }D,
			\\u &=& f &\mbox{on }\partial D,
		\end{array}\right.
	\end{equation}
	where $f: \partial D\rightarrow \mathbb{R}_{+}$ is  continuous  and $g:D\rightarrow \mathbb{R}_{+}$ is Borel measurable locally bounded. It will be convenient to extend the function $\psi$ to $\mathbb{R}$ by setting $\psi(t)=0$ if $t\leq 0.$  By a nonnegative solution in $D$ to equation
	\begin{equation}\label{peq}
		-Lu+\xi\psi(u) = g,
	\end{equation} 
	we mean a continuous function $u: D\to\mathbb{R}_+$ satisfying~\eqref{peq} 
	in  the distributional sense,  i.e, for all  $\varphi \in \mathcal{C}_{c}^{\infty}(D)$,
	\begin{equation}
		-\int_{D}uL^{*}\varphi\, d\lambda_d+\int_{D}\xi\psi(u)\varphi \, d\lambda_d = \int_{D}g\varphi \, d\lambda_d. 
	\end{equation} 
	The boundary condition in problem~\eqref{base} means that  for all $z\in\partial D$,
	\begin{equation}
		\lim_{x\in D, \, x\rightarrow z} u(x) =f(z).
	\end{equation} 
	
	\begin{proposition}\label{lemcom}
		Let $\Psi:\mathbb{R}\to \mathbb{R}_+$ be  nondecreasing and let $u, v\in\mathcal{B}(D)$ be two locally bounded functions such that $G_{D}(\xi\Psi(u))$ and $G_{D}(\xi\Psi(v)))$ are not identically infinite in~$D$. Assume that 
		$$s:=u-v+G_{D}(\xi\Psi(u))-G_{D}(\xi\Psi(v))$$ is a nonnegative $L-$superharmonic function in $D$. Then $u\geq v$ in $D$.
	\end{proposition}
	\begin{proof}
		We first observe that $G_Dw$ is finite and continuous in~$D$ by Theorem~\ref{theoo}, where  $w:=\xi\Psi(u)-\xi\Psi(v)$. 
		The fact that $\Psi$ is nondecreasing yields that  $\{w^+>0\}\subset \{u>v\}$ and consequently  
		\begin{equation}\label{eqcomp}
			s+G_{D}w^-\geq G_{D}w^+ \mbox{ in } \{w^+>0\}.
		\end{equation}
		We define $p_n:=n\wedge w^+\textbf{1}_{K_n}$  where $(K_n)_{n\geq 1}$ is a sequence  of compact sets satisfying 
		$$
		K_n\subset K_{n+1}\subset \{w^+>0\} \mbox{ for all }n\geq 1\quad\mbox{and}\quad \lambda_d(\cup_{n\geq 1}K_n)=\lambda_d(w^+>0).
		$$
		According to   Theorem~\ref{theoo},  for every $n\geq 1$,  $G_Dp_n$ is $L$-harmonic in  $U_n:=D\as K_n $ and consequently $\theta_n:=s+G_Dw^--G_Dp_n$  is   $L-$superharmonic in~$U_n$. 
		On the other hand, since  $p_n$ is bounded and $D$ is regular, Proposition~\ref{pror} yields that $\lim_{x\in D,x\to z}G_D p_n(x)=0$ for every $z\in \partial D$. Then, 
		for every $z\in\partial U_n\cap\partial D$,
		$$
		\liminf_{x\in U_n,\, x\to z}\theta_n(x)=\liminf_{x\in U_n,\, x\to z}(s+G_Dw^-)(x)\geq 0 
		$$
		Moreover, since $K_n\subset\{w^+>0\}$, it follows from~\eqref{eqcomp} that  $\theta_n \geq 0$ in $K_n$. Hence, seeing that $\theta_n$ is lower semicontinuous in $D$, we deduce that  for every  $z\in \partial U_n\cap D\subset K_n$,  $\liminf_{x\in U_n,\,x\to z}\theta_n(x)\geq 0$.  
		We then conclude by Proposition~\ref{prinmin} that for every $n\geq 1$,
		\begin{equation*}
			s+G_Dw^--G_Dp_n\geq 0 \mbox{ in }U_n.
		\end{equation*}
		Letting $n$ tend to $\infty$, we obtain that $u-v=s+G_Dw^- -G_Dw^+\geq 0$ in $D\as\{w^+>0\}$ which yields that $u-v\geq 0$ in $D$.
	\end{proof}
	\begin{theorem}\label{thexis}
		Let	$s\in\mathcal{C}^+(D)$ be  $L-$superharmonic in $D$ such that $G_D(\xi\psi(s))\not\equiv \infty$. Then there exists one and only one function  $u\in \mathcal{C}^+(D)$ such that 
		\begin{equation}\label{qe}
			u+G_D(\xi\psi(u))=s\mbox{ in }D.
		\end{equation}
	\end{theorem}
	\begin{proof}
		Due to the previous proposition, the integral equation \eqref{qe} has at most one solution.  The existence of a solution  will be established  in two steps.
		
		{\em Step 1.} Consider first an open set $V\Subset D$. Let $a=\sup_{x\in \overline{V}}|s(x)|$ and define
		$$
		\Lambda(u)=s-G_V(\xi{\psi_a}(u)) \quad\mbox{for all }u\in\mathcal{B}_b(V),
		$$	
		where 	${\psi_a}(t)= \psi(a\wedge t)$ for every $t\in\mathbb{R}$. 
		We  claim that  $\Lambda$ is continuous in $\mathcal{B}_b(V)$ endowed with its uniform norm.  Indeed, let $(u_n)_{n\geq 1}\subset \mathcal{B}_b(V)$ be a sequence converging   to a function  $u\in \mathcal{B}_b(V)$. Then, due to the dominated convergence theorem, $(G_V(\xi{\psi_a}(u_n))_{n\geq 1}$ is pointwise  convergent to $G_V(\xi\psi_a(u))$.   Assume that the last convergence does not hold uniformly. Thus, there exist $\varepsilon>0$ and a subsequence $(u_{n_k})_{k\geq 1}$ of $(u_n)_{n\geq 1}$
		such that 
		\begin{equation}\label{unk}
			\sup_{x\in V}\left|G_V(\xi{\psi_a}(u_{n_k}))(x)-G_V(\xi{\psi_a}(u))(x)\right|\geq \varepsilon \quad\mbox{for all }k\geq 1.
		\end{equation}
		Since the operator  $G_V$ is compact  and  the sequence $(\xi{\psi_a}(u_{n_k}))_{k\geq 1}$ is  bounded  in $\mathcal{B}_b(V)$, we can find a subsequence $(w_k)_{k\geq 1}$ of  $(u_{n_k})_{k\geq 1}$ such that 
		$(G_V(\xi{\psi_a}(w_k)))_{k\geq 1}$ converges to $G_V(\xi{\psi_a}(u))$ with respect to the uniform norm in $\mathcal{B}_b(V)$. This yields a contradiction with~\eqref{unk} above. Hence $\Lambda$ is continuous in $\mathcal{B}_b(V)$ and the claim is proved. Using again the fact that the operator $G_V$ is compact, we check out that $ \Lambda(\mathcal{B}_b(V)) $ is  relatively compact in $\mathcal{B}_b(V)$. Then, we conclude  in virtue of the Schauder's fixed point theorem, that there exists $v\in \mathcal{B}_b(V)$ such that $\Lambda(v)=v$. On the other hand, using Proposition~\ref{lemcom}, it is clear that $v\geq 0$ in $V$ and thereby $0\leq v \leq s\leq a$ in $V$. This yields that ${\psi_a}(v)=\psi(v)$ and consequently  
		\begin{equation}\label{eqcom}
			v+G_V(\xi\psi(v))=s\quad\mbox{in } V.
		\end{equation}	
		Observe that $v$ is continuous in $V$  since   $G_V(\xi\psi(v))$ is continuous in $V$  by   Proposition~\ref{pror}.(a).    
		
		{\em Step 2.} Choose a sequence $(D_{n})_{n\geq 1}$ of regular open sets such that 
		\begin{equation*}
			D_n\Subset D_{n+1} \quad\mbox{for all }n\geq 1\quad \mbox{ and }\cup_{n\geq 1}D_n=D.
		\end{equation*}
		Let, for every $n\geq 1$,  $u_{n}$ be  the solution of 
		\begin{equation}\label{eq2}
			u_{n}+G_{D_{n}}(\xi\psi(u_{n}))=s  \quad \mbox{in }D_{n}.
		\end{equation}
		Then, in virtue of formula~\eqref{ggg} we observe that 
		$$
		u_{n+1}+G_{D_n}(\xi\psi(u_{n+1}))=s-H_{D_n}G_{D_{n+1}}(\xi\psi(u_{n+1})) \quad\mbox{in }D_n
		$$
		and hence
		$
		u_n-u_{n+1}+G_{D_n}(\xi\psi(u_n))-G_{D_n}(\xi\psi(u_{n+1}))
		$
		is a nonnegative $L-$harmonic function in $D_n$. So, in virtue of  Proposition~\ref{lemcom}, we have  $0\leq u_{n+1}\leq u_n\leq s$ in $D_n$ for all $n\geq 1$. Let $u:D\to{\mathbb{R}}_+$ be the function given for every $x\in D$ by 
		$$
		u(x)=\lim_{n\to\infty}u_n(x)=\inf_{n\geq 1} u_{n}(x).
		$$
		Seeing that  $G_D(\xi\psi(s))<\infty$ in $D$ and  applying  the dominated convergence theorem, we get for every $x\in D$ that 
		$$
		\lim_{n\rightarrow \infty}\int_{D_{n}}G_{D_{n}}(x,y)\xi(y)\psi(u_{n}(y))dy  =  \int_{D}G_{D}(x, y)\xi(y)\psi(u(y))dy.
		$$
		Thus, letting $n$ tend to infinity   in~\eqref{eq2} we conclude that $u$ satisfies~\eqref{qe}.  The continuity of $u$ follows obviously from the fact that $G_D(\xi\psi(u))$ is continuous in~$D$. 
	\end{proof}
	
	For every $f\in\mathcal{B}^+(\partial D)$ and every locally bounded function $g\in\mathcal{B}^+(D)$ such that $H_Df$ and $G_Dg$ are not identically infinite in~$D$, it 
	follows from Proposition~\ref{lem01} and Theorem~\ref{theoo} that the function  
	\begin{equation*}
		S_D(f,g)=H_Df +G_Dg
	\end{equation*}
	is   continuous and $L-$superharmonic in $D$. Then, assuming that 
	$G_D(\xi\psi( S_D(f,g)) )\not\equiv\infty$ in~$D$,
	the previous theorem yields that the integral  equation
	\begin{equation}\label{s2}
		u +\int_D G_D(\cdot,y)\xi(y)\psi(u(y))\,dy= S_D(f,g)
	\end{equation}
	admits a unique   solution  $u\in \mathcal{C}^+(D)$. In all the following, we use the notation 
	\begin{equation}
		u=U_D^\xi(f,g).
	\end{equation}
	For $i=1,2,$ let $f_i\in\mathcal{B}^+(\partial D)$ and $g_i,\xi_i \in\mathcal{B}^+(D)$ be locally bounded in~$D$ such that  $s_i=S_D(f_i, g_i)<\infty$ and $G_D(\xi_i\psi(s_i)))<\infty$ in~$D$.  If $f_1\leq f_2$, $g_1\leq g_2$ and $\xi_1\geq \xi_2$, it follows immediately from Proposition~\ref{lemcom} that 
	\begin{equation}\label{compf1f2}
		U_D^{\xi_1}(f_1, g_1)\leq U_D^{\xi_2}(f_2,g_2). 
	\end{equation}

	\begin{proposition}\label{copro}
		We consider a locally bounded function $g\in\mathcal{B}^+(D)$  and let  $f\in\mathcal{B}^{+}(\partial D)$	 such that $s=S_D(f,g)<\infty$ and $G_D(\xi\psi(s))<\infty$ in~$D$. Let $(\xi_n)_{n\geq 1}, (g_n)_{n\geq 1}\subset \mathcal{B}^+(D) $ and $(f_n)_{n\geq 1}\subset \mathcal{B}^{+}(\partial D) $  be nondecreasing sequences which are   pointwise convergent respectively to $\xi,g$ and~$f$. Then:
		\begin{enumerate}
			\item[a)]  $U_D^{\xi_n}(f, g)$ decreases to $U_D^{\xi}(f, g)$.
			\item[b)] 			$U_D^{\xi}(f_n, g_n)$ increases to $U_D^{\xi}(f, g)$.
		\end{enumerate}
	\end{proposition}
	\begin{proof}
		We define $u_n:=U_D^{\xi_n}(f, g)$ for every  $n\geq 1$. Applying~\eqref{compf1f2}, we get that $(u_n)_{n\geq 1}$ is nonincreasing in $D$ and thereby it converges to $u:=\inf_{n\geq 1} u_n$. Furthermore, applying the dominated convergence theorem, we obtain 
		$$
		\lim_{n\to\infty} \int_{D}G_D(\cdot,y)\xi_{n}(y)\psi(u_n(y))\,dy=\int_{D}G_D(\cdot,y)\xi(y)\psi(u(y))\,dy.
		$$
		Consequently, $u+G_D(\xi\psi(u))=s$ in $D$ which yields that $u=U_D^\xi(f,g)$ and statement~(a) is proved.  Statement~(b) is similar.
	\end{proof}

	Under appropriate conditions, the following theorem ensures that  $U_D^\xi(f,g)$ is none other than   the solution to the semilinear boundary value problem~\eqref{base}.
	
	\begin{theorem}\label{theopro}
		Assume that $G_D \xi, G_D g\in\mathcal{C}_{0}(D)$. Then, for every function $f\in\mathcal{C}^+(\partial D)$,  problem~\eqref{base} admits one and only one solution $u\in\mathcal{C}^+(D)$. More precisely  $u=U_D^\xi(f,g)$.  
	\end{theorem}
	
	\begin{proof}
		It is sufficient to prove that $u$ is a solution to~\eqref{s2} if and only if $u$ is the solution to problem~\eqref{base}. Let $a=\sup_{x\in D}|S_D(f,g)|$ and let $u=U_D^\xi(f,g)$ be the solution to \eqref{s2}. It is easy to see   $0\leq u \leq a$ in $D$ which yields that 
		$
		G_{D}(\xi\psi(u))\leq \psi(a) G_{D}\xi<\infty
		$ and thereby $G_D(\xi\psi(u))\in \mathcal{C}(D)$ and according to  formula~\eqref{gdp},
		$$
		-LG_{D}(\xi\psi(u))=\xi\psi(u) \quad\mbox{in }D.
		$$ 
		We also have $-LG_{D}g=g$ in $D$ since $G_Dg$ is continuous. Therefore, $u\in\mathcal{C}^+(D)$ and
		$$
		-Lu+\xi\psi(u)=g \quad\mbox{in }D.
		$$ On the other hand, for every $z\in\partial D$,
		$
		0\leq \lim_{x\to z}G_D (\xi\psi(u))(x)\leq \psi(a)\lim_{x\to z} G_{D}\xi=0
		$
		and consequently $\lim_{ x\rightarrow z} u(x)=f(z)$ by~\eqref{s2}.
		Thus $u=U_D^\xi(f,g)$ is a  solution to the boundary problem~\eqref{base}.
		Conversely, let $u$ be a solution to~\eqref{base}. Using the fact that $G_D\xi$ and $G_Dg$ are in the class $\mathcal{C}_0(D)$, we easily check out that 
		$$
		h:=u+G_D(\xi\psi(u))-G_Dg
		$$
		is $L-$harmonic in $D$ and $\lim_{x\to z}h(x)=f(z)$ for all $z\in\partial D$. Thus $h=H_Df$ and $u$ satisfies~\eqref{s2}.	
	\end{proof}
	
	In virtue of the previous theorem, we observe that for any function $u\in\mathcal{C}^+(D)$,  the following three statements are equivalent:
	\begin{enumerate}
		\item[i)] The function $u$ is a solution to equation~\eqref{peq} in $D$.
		\item[ii)] $U_V^\xi(u,g)=u$ in $V$ for every regular set $V\Subset D$.
		\item[iii)]	$U_B^\xi(u,g)=u$ in $B$ for every ball $B\Subset D$.
	\end{enumerate}

	\begin{proposition}\label{theo3}  Assume that   $\xi\in\mathcal{C}^\alpha(D)$,  $\psi\in\mathcal{C}^\alpha(\mathbb{R}_+)$ and let $f\in\mathcal{B}(\partial D)$, $g\in\mathcal{C}^\alpha(D)$ be nonnegative functions such that $S_D(f,g)<\infty$ and $G_D\big(\xi\psi(S_D(f,g))\big)<\infty$ in~$D$. Then  $U_D^\xi(f,g)$ is twice continuously  differentiable in~$D$.  
	\end{proposition}
	
	\begin{proof} Consider a ball    $B\Subset D$. 	 Notice that $u=U_D^\xi(f,g)$ is continuous in~$D$ and, in virtue of~\eqref{s2} and~\eqref{ggg}, we have 
		$$
		u=H_{B}u+G_{B}g-G_{B}\big(\xi \psi(u)\big)\quad \mbox{in }B.
		$$ 
		We first  observe that $H_{B}u\in \mathcal{C}^{2}(B)$ by the hypoellipticity property of~$L$ and $G_Bg\in \mathcal{C}^{2}(B)$ in view of~(b) in Proposition~\ref{pror}. Next, since $p=\xi\psi(u)$ is bounded in~$B$,  statement~(a) in Proposition~\ref{pror} yields that  $G_Bp\in\mathcal{C}^\alpha(B)$ and consequently $u\in \mathcal{C}^\alpha(B)$. By simple computation, we check out that  $p\in \mathcal{C}^{\alpha^2}(B)$ and thereby $G_Bp\in \mathcal{C}^{2}(B)$ by Proposition~\ref{pror}(b). Whence $u\in \mathcal{C}^2(B)$ for any arbitrary ball $B\Subset D$ which means that $u\in\mathcal{C}^2(D)$. 
	\end{proof}
	
	Let us note  that all  results obtained  in this section can easily be extended to a more general setting, where we consider instead of  functions $\xi(x)$ and $\psi(t)$, a function 
	$$
	\Phi: (x,t)\in  D\times \mathbb{R}_{+} \mapsto \Phi(x,t)\in \mathbb{R}
	$$ 
	satisfying some hypotheses analogous to those assumed on  $\xi$ and $\psi$. Besides, it should be specified that the fact that $\xi$ and $g$ are locally bounded is not really necessary in our approach. In fact, it is enough to assume that $\xi$ and $g$ belong to the local Kato class \cite{chung1995brownian}.
	
	\section{Estimates of solutions}

	This section is devoted to establish our main results which consist in  lower bound estimates of  nonnegative  solutions to equation~\eqref{peq} and more general to inequality 
	\begin{equation}\label{sup-sol0}
		-Lu+\xi\psi(u) \geq  g \mbox{ in }D.
	\end{equation} 
	More precisely, under appropriate conditions on $\psi$, we will determine a function $\varphi$ with values in  $]0,1]$  such that 	for every   $f\in\mathcal{C}^+(\partial D)$, $g\in\mathcal{B}^+(D)$ and for all $x\in D$,	
	\begin{equation}\label{eqcp}
		U_D^\xi(f,g)(x)\geq	s(x)\varphi\left(\frac{G_D(\xi\psi(s))(x)}{s(x)}\right).
	\end{equation}
	The function  $g$ is assumed to be locally bounded  in~$D$ such that 	$G_Dg\in\mathcal{C}_0(D)$ and the $L-$superharmonic function 
	$s=S_D(f,g)>0$ in $D$.
	
	We are only interested  in locally bounded functions~$\xi\in\mathcal{B}(D)$ which are nonnegative. Analogous problems  with $\xi\leq 0$  are recently investigated by  Grigor'yan and Verbitsky in~\cite{grigor2020}. They gave a real positive function $\varrho$, depending only on~$\psi$, such that if  $S_D(f,g)\geq 1$ in $D$, then the  solution $u$ to problem~\eqref{base}  satisfies for all $x\in D$,
	\begin{equation}\label{urho}
		u(x)\geq \varrho\left(G_D\xi \right(x)).
	\end{equation}
	The particular case where $L=\Delta$ is the Laplacian operator and $f,g$ are identically zero, is already studied by  M\^aagli and Zribi \cite{zribi2001}.
	For signed function $\xi$,   problem~\eqref{base} is  examined  by  Grigor'yan and Verbitsky  \cite{grigor2019pointwise} in the special case where $\psi(t)=t^\gamma$ with $\gamma>0$.  Assuming that $p=G_Dg>0$ in $D$, they showed that  if  $u\in \mathcal{C}^2(D)$  is a nonnegative solution to 
	\begin{equation}\label{sup-sol-g}
		-Lu+\xi u^\gamma \geq  g \mbox{ in }D,
	\end{equation}     
	then 	 for every $x\in D$,
	\begin{equation}\label{estgr}
		u(x)\geq \left\{\begin{array}{ll}
			p(x)\left(1+(\gamma-1)\frac{G_D\left(p^\gamma \xi\right)(x)}{p(x)}\right)_+^{1/(1-\gamma)} & \mbox{if }\gamma<1,
			\\ p(x)\exp\left({-\frac{G_D(\xi p)(x)}{p(x)}}\right) & \mbox{if }\gamma=1, 
			\\ p(x)\left(1+(\gamma-1)\frac{G_D(p^\gamma \xi)(x)}{p(x)}\right)^{1/(1-\gamma)} & \mbox{if }\gamma>1.
		\end{array}\right.
	\end{equation}	
	
	In this paper, we establish similar estimates for a large class of functions $\psi$ which contains   polynomial functions, logarithmic functions as well as  exponential functions. More precisely, we consider a continuous function  $\psi:\mathbb{R}_{+}\to \mathbb{R}_{+}$ such that  $\psi(0)=0$,  $\psi(t)>0$ for all $t>0$, $\psi$ is continuously differentiable in $ ]0,\infty[$ and 
	\begin{equation}\label{hypcom}
		\psi(rt)\leq c\psi(r)\psi(t)\quad \mbox{for all }r\in[0,1], t\in\mathbb{R}_{+} 
	\end{equation}
	where $c$ is a positive constant not depending on either $s$ or $t$. Let $0<\ell\leq \infty$ given by 
	\begin{equation*}
		\ell c=\int_{0}^{1}\frac{1}{\psi(s)}\,ds
	\end{equation*}
	and define the function  $\Theta:]0,1]\rightarrow [0, \ell[$  for every $t\in ]0,1]$ by 
	$$
	\Theta(t)=\int_{t}^{1}\frac{1}{c\psi(s)}ds.
	$$
	Then, obviously   $\Theta$ admits an inverse function which will be denoted by $\varphi$ and extended  to~$\mathbb{R}_+$ by setting $\varphi(t)=0$ if $t\geq \ell$. Hence, $\varphi$ is continuous and nonincreasing in  $\mathbb{R}_+$, twice continuously  differentiable in $]0,\ell[$  and for every $r\in ]0,\ell[$, 
	$$
	\varphi'(r)=-c\psi(\varphi(r))<0\quad\mbox{ and }\quad \varphi''(r)=c^2\psi'(\varphi (r))\psi(\varphi (r))> 0.
	$$ 
	
	In order to prove Theorem~\ref{thp} below, we first establish some technical lemmas. 
	\begin{lemma}\label{lem1}
		Let $u:D\rightarrow \mathbb{R}_+$ and $v:D\rightarrow ]0,\ell[$ be  twice continuously differentiable  functions. If $u$ is $L-$superharmonic in $D$, then
		$$
		L(uv)\geq 	\frac{Lu- L(u\varphi(v))}{c\psi(\varphi(v))}.
		$$
	\end{lemma}
	\begin{proof} Obviously  $w=\vphi(v)\in\mathcal{C}^2(D)$ and by elementary  computation, we get that  
		$$
		L w=\varphi'(v)Lv+\varphi''(v)<A\nabla v,\nabla v>.
		$$
		Consequently 
		\begin{eqnarray*}
			L(uw)&=&uLw+wLu+<\nabla u, A\nabla w>+<A\nabla u,\nabla w>\\
			&=&\varphi'(v)uLv+\varphi''(v)u<A\nabla v,\nabla v>+\varphi(v)Lu\\
			&&\qquad+\varphi'(v)<\nabla u, A\nabla v>+\varphi'(v)<A\nabla u, \nabla v>\\
			&=&\varphi(v)Lu-\varphi'(v) vLu+ \varphi''(v)u<A\nabla v,\nabla v>\\
			&&\qquad +\varphi'(v)\left( uLv+vLu+ <\nabla u, A\nabla v>+<A\nabla u, \nabla v>\right)\\
			&=&\varphi(v)Lu-\varphi'(v) vLu+ \varphi''(v)u<A\nabla v,\nabla v>+\varphi'(v) L(uv)
		\end{eqnarray*}
		Since the matrix $A$ is positive-definite and the function $\varphi$ is convex, it follows that 
		\begin{eqnarray*}
			&&	L(u\varphi(v))-\left(\varphi(v)-v\varphi'(v)\right)Lu-\varphi'(v)L(uv) \\
			&& \qquad\qquad\qquad = \quad \varphi''(v)u<A\nabla v,\nabla v>\\
			&& \qquad\qquad\qquad \geq \quad 0.	
		\end{eqnarray*}
		Hence
		\begin{eqnarray*}
			-\varphi'(v)L(uv)  & \geq &  -L(u\varphi(v))+\left(\varphi(v)-v\varphi'(v)\right)Lu
			\\ & = &
			-L(u\varphi(v))+\left(\varphi(v)-1-v\varphi'(v)\right)Lu+Lu
		\end{eqnarray*}
		In virtue of the mean value theorem and the fact that $\varphi'$ is nondecreasing, we obtain that
		$
		\varphi(v)-1-v\varphi'(v)\leq 0
		$ in $D$. 
		Assume that $u$ is $L-$superharmonic in $D$. Then we deduce from  \cite[Theorem~3.5]{gilbarg1998elliptic} that $Lu\leq 0$ in $D$ and consequently  $\left(\varphi(v)-1-v\varphi'(v)\right)Lu\geq 0$ in $D$.
		Finally, using the fact that $\varphi'(v)=-c\psi(\varphi(v))\leq 0 $, we get that
		$$
		L(uv)\geq 	\frac{Lu- L(u\varphi(v))}{c\psi(\varphi(v))}.
		$$
	\end{proof}
	
	\begin{lemma}\label{lem2}
		Let $f\in\mathcal{B}_b^+(\partial D)$ and $g\in\mathcal{B}_b^+(D)$ such that $s=S_D(f,g)>0$ in $D$. Assume that  $\xi$ is bounded and let  $(\xi_n)_{n\geq 1}$, $(g_n)_{n\geq 1}\subset \mathcal{B}^{+}(D)$ and $(f_n)_{n\geq 1}\in\mathcal{B}^+(\partial D)$  be bounded sequences which are pointwise convergent respectively to $\xi$, $g$ and $f$. Assume that for all $n\geq 1$, $s_n=S_D(f_n,g_n)>0$ and 
		\begin{equation*}
			U_D^{\xi_n}(f_n,g_n)\geq s_n\varphi\left(\frac{G_D(\xi_n\psi(s_n))}{s_n}\right)\quad\mbox{in }D.
		\end{equation*}	
		Then  \eqref{eqcp}  holds true.
	\end{lemma}
	\begin{proof} We have $u_n+G_Dp_n=s_n$ where  $u_n=U_D^{\xi_n}(f_n, g_n)$  and $p_n=\xi_n\psi(u_n)$  for all $n\geq 1$. Since $(p_n)_{n\geq 1}$ is a bounded  sequence, it follows from Proposition~\ref{pro2} that there exists a subsequence $(p_{n_k})_{k\geq 1}$ such that 		$(G_Dp_{n_k})_{k\geq 1}$  is uniformly convergent in $D$. Seeing that  $(s_{n_k})_{k\geq 1}$ is pointwise convergent to $s$, we conclude that $(u_{n_k})_{k\geq 1}$ converges in $D$ to a   function $u\in\mathcal{B}_b^+(D)$.  Applying the dominated convergence theorem, we obtain that for every $x\in D$,
		$$
		\lim_{k\to\infty}G_Dp_{n_k}(x)=G_D(\xi\psi(u))(x).
		$$
		Therefore  $ u+G_D(\xi\psi(u))=s$ in $D$, which yields in view of Theorem~\ref{thexis} that $u=U_D^{\xi}(f,g)$.   Again, since sequences $(\xi_n)_{n\geq 1}$ and  $(s_n)_{n\geq 1}$ are bounded, we see that  $(G_D(\xi_{n_k}\psi(s_{n_k})))_{k\geq 1}$  is pointwise convergent to $G_D(\xi\psi(s))$. Hence, letting $k$ tend to infinity in the formula
		$$
		u_{n_{k}}\geq 	s_{n_k}\varphi\left(\frac{G_D(\xi_{n_k}\psi(s_{n_k}))}{s_{n_k}}\right) \quad\mbox{in }D,
		$$
		we obtain \eqref{eqcp} in virtue of the continuity of $\varphi$.
	\end{proof}
	
	\begin{theorem}\label{thp} 	
		Let $g\in\mathcal{B}^+(D)$ be locally bounded and  let $f\in\mathcal{B}^+(\partial D)$  such that $0<s<\infty$ and $G_D(\xi\psi(s))<\infty$ in $D$ where   
		$s=S_D(f,g)$. 
		Then 
		\begin{equation*}
			s\,\varphi\left(\frac{G_{D}(\xi\psi(s))}{s}\right)\leq U_D^\xi(f,g) \leq s \quad \mbox{ in } D. 
		\end{equation*}
	\end{theorem}
	\begin{proof} Seeing that $u\leq s$ is immediate with $u= U_D^\xi(f,g)$, we only need to prove that 
		\begin{equation}\label{estf}
			\frac{u}{s}\geq \varphi\left(\frac{G_D(\xi\psi(s))}{s}\right)\quad \mbox{in } D.
		\end{equation}
		
		{\em Step 1.} We first assume that $\xi$, $g$ are bounded and $\alpha-$H\"older continuous in $D$ and that~$f$ is bounded in~$\partial D$. Then 
		$s=H_Df +G_D g\in\mathcal{C}^2(D)$ since  $H_Df\in\mathcal{C}^2(D)$ by the  hypoellipticity property of~$L$ and  $G_Dg \in\mathcal{C}^2(D)$ in view of Proposition~\ref{pror}.(b).  Besides, $u\in\mathcal{C}^2(D)$  in virtue of  Proposition~\ref{theo3} and the fact that $\xi$ and $\psi$ are $\alpha-$H\"older continuous.
		Therefore, for every $\ve>0$, $v_{\ve}=\Theta\left(\frac{u+\varepsilon}{s+\varepsilon}\right)$ is twice  continuously differentiable in $D$ with values in $]0,\ell[$ and consequently using  Lemma~\ref{lem1} we obtain that
		\begin{eqnarray*}
			L((s+\ve)v_\ve)
			\geq \frac{L(s+\ve)-L(u+\ve)}{c\psi(\frac{u+\ve}{s+\ve})}
			=-\xi\frac{\psi(u)}{c\psi(\frac{u+\ve}{s+\ve})}\geq -\xi \psi(s).
		\end{eqnarray*}
		The last inequality follows from~\eqref{hypcom} and the fact that $\psi$ is nondecreasing in~$\mathbb{R}_+$. We also see that $G_D(\xi\psi(s))\in \mathcal{C}^2(D)$ by Proposition~\ref{pror}.(b) and therefore
		\begin{equation*}
			w_\ve=(s+\ve)v_\ve - G_D(\xi\psi(s))\in \mathcal{C}^2(D) \quad\mbox{and} \quad L w_\ve\geq 0\mbox{ in }D.
		\end{equation*}
		On the other hand, for every $z\in \partial D$,   $\lim_{x\to z}G_D(\xi\psi(s))(x)=0$ by Proposition~\ref{pror}.(a). Moreover, seeing that
		$$
		1\leq \frac{s+\ve}{u+\ve}= \frac{u+\ve+G_D(\xi\psi(u))}{u+\ve}\leq 1+\frac{G_D(\xi\psi(s))}{\ve}\quad\mbox{in }D,
		$$
		we also obtain that  $\lim_{x\to z}v_\ve(x)=\Theta(1)=0$ and thereby  $\lim_{x\to z}w_\ve(x)=0$ for all $z\in \partial D$. We then conclude  in view of the maximum principle \cite[Theorem 3.1]{gilbarg1998elliptic} that $w_\ve\leq 0$ in $D$ and consequently 
		\begin{equation*}
			\frac{u+\ve}{s+\ve}\geq \varphi\left(\frac{G_D(\xi\psi(s))}{s+\ve}\right)\quad \mbox{ in } D.
		\end{equation*}
		Letting $\ve$ tend to $0$, we conclude that~\eqref{estf} holds true whenever  $\xi$ and $g$ are  bounded  $\alpha-$H\"older continuous in $D$ and that~$f$ is bounded in~$\partial D$.

		{\em Step 2.} We consider at this step the case where $\xi$ and $g$ are only bounded in $D$ and we still assume that  $f$ is   bounded in $\partial D$. In virtue of  Corollary~4.23 and Theorem~4.9 in~\cite{brezis2010functional}, we choose two sequences $(\xi_n)_{n\geq 1}, (g_n)_{n\geq 1}\subset \mathcal{C}_c^{\infty}(D) $ such that $0\leq \xi_n\leq \|\xi\|_{\infty}$, $0\leq g_n\leq \|g\|_{\infty}$ for all $n\geq 1$ and 
		$$
		\lim_{n\to\infty}\xi_n= \xi,\; \lim_{n\to\infty}g_n= g\;\mbox{a.e in }D.
		$$
		Seeing that $\lim_{n\to\infty}S_D(f,g_n)(x)=s(x)$ and $s(x)>0$ for all $x\in D$, we can impose on the $L$-superharmonic function  $s_n=S_D(f,g_n)$ to be positive everywhere in $D$ for every $n\geq 1$.  It follows from Step~1 above that, for all $n\geq 1$,  
		$$
		\frac{U_D^{\xi_n}(f,g_n)}{s_n}\geq \varphi\left(\frac{G_D(\xi_n\psi(s_n))}{s_n}\right)\quad\mbox{in }D.
		$$
		Hence, \eqref{estf} holds true by Lemma~\ref{lem2}.
		
		{\em Step 3.} We return to the general setting where functions  $\xi, g$  are nonnegative locally bounded in~$D$ and $f\in\mathcal{B}^+(\partial D)$. We define for every $n\geq 1$,    $\xi_n=\xi\wedge n$, $g_n=g\wedge n$, $f_n=f\wedge n$ and let  $ s_n=S_D(f_n,g_{n})>0$ in~$D$.   According to Step~2,  for every $n\geq 1$ and every $k\geq 1$,
		\begin{equation}\label{estk}
			\frac{U_D^{\xi_k}(f_n,g_n)}{s_n}\geq \varphi\left(\frac{G_D(\xi_k\psi(s_n))}{s_n}\right)\quad\mbox{in }D.
		\end{equation}
		On the other hand, for every $k\geq 1$, $\lim_{n\to\infty}U_D^{\xi_k}(f_n,g_n)=U_D^{\xi_k}(f,g)$  by  Proposition~\ref{copro}.(b) and  $\lim_{n\to\infty}G_D(\xi_k\psi(s_n))=G_D(\xi_k\psi(s))$  by the monotone convergence theorem. Hence, letting $n$ tend to infinity in \eqref{estk},  we get that for all $k\geq 1$,
		\begin{equation*}
			\frac{U_D^{\xi_k}(f,g)}{s}\geq \varphi\left(\frac{G_D(\xi_k\psi(s))}{s}\right)\quad\mbox{in }D.	
		\end{equation*} 
		Letting $k$ tend to infinity, we conclude that~\eqref{estf} holds true in virtue of the continuity of  $\varphi$. Observe that  $\lim_{k\to\infty}U_D^{\xi_k}(f,g)=u$  by Proposition~\ref{copro} and $\lim_{k\to\infty}G_D(\xi_k\psi(s))=~G_D(\xi\psi(s))$ by the monotone convergence theorem.
		
	\end{proof}
	
	Now, we  use our previous results  to establish a lower bound estimate of nonnegative solutions to inequality
	\begin{equation}\label{sup-sol}
		-Lu+\xi\Psi(u) \geq  g \quad\mbox{in }D,
	\end{equation} 
	where $\Psi:\mathbb{R}_+\to\mathbb{R}$ is any locally bounded Borel measurable function which is  bounded above  by some function $\psi$ satisfying hypotheses given at the beginning of this section.  Here, we do not need to assume that $\Psi$ is monotone nor that it keeps a constant sign.   We point out that a solution $u$ to inequality~\eqref{sup-sol} should be understood in the distributional sense, that is,  $u:D\to\mathbb{R}$ is continuous  and for every nonnegative function   $\varphi \in \mathcal{C}_{c}^{\infty}(D)$,
	\begin{equation*}
		-\int_{D}uL^{*}\varphi\, d\lambda_d+\int_{D}\xi\Psi(u)\varphi \, d\lambda_d \geq  \int_{D}g\varphi \, d\lambda_d.
	\end{equation*}
	\begin{theorem}\label{thecp}
		Let  $\Psi$ and $\psi$ be as above,  $D$ be any Greenian domain of $\mathbb{R}^d$  (not necessarily bounded) and let  $g\in \mathcal{B}^+(D)$ be locally bounded such that $0<p<\infty$ in $D$ where $p=G_D g$. Then, for every nonnegative solution $u$ to inequality~\eqref{sup-sol} such that   $Lu$ is a locally bounded function in $D$,    
		\begin{equation}\label{estg}
			u\geq p\, \varphi\Big(\frac{G_D (\xi\psi(p))}{p}\Big) \quad\mbox{in }D.
		\end{equation}		
	\end{theorem}
	\begin{proof} 	 Let  $u$ be a  solution to inequality~\eqref{sup-sol}, define $g_0:=-Lu+\xi\psi(u)$ and choose a sequence  $(D_{n})_{n\geq 1}$  of regular bounded   domains such that $D_n\Subset D_{n+1}$ for all $n\geq 1$ and $\cup_{n\geq 1}D_n=D$. For every $n\geq 1$, define $v_n:=U_{D_n}^\xi(0,g)$, that is,  
		\begin{equation*}
			\left\{\begin{array}{rcll}
				-Lv_n+\xi\psi(v_n) &=& g &\mbox{in }D_n,
				\\ v_n &=& 0&\mbox{on }\partial D_n.
			\end{array}\right.
		\end{equation*}	
		Since $u\geq 0$ and since $g_0\geq -Lu+\xi\Psi(u)\geq  g$ in $D$, it follows immediately from~\eqref{compf1f2} that $	u=U_{D_n}^\xi(u,g_0)\geq  v_n$ in~$D_n$. On the other hand, we obviously see that the sequence  $(p_n)_{n\geq 1}$ of $L$-superharmonic functions in~$D_n$ given by  $p_n=G_{D_n}g=S_{D_n}(0,g)$, is nondecreasing and pointwise convergent to~$p>0$ in~$D$. Then, there exists a subsequence $(p_{n_k})_{k\ge 1}$ such that $p_{n_k}>0$ in $D_{n_k}$ for every $k\geq 1$.  We then apply Theorem~\ref{thp} in the  regular domain~$D_{n_k}$ to obtain 
		\begin{equation*}
			v_{n_k}\geq p_{n_k} \vphi\left(\frac{G_{D_{n_k}}(\xi\psi(p_{n_k}))}{p_{n_k}}\right)\quad\mbox{in }D_{n_k}. 
		\end{equation*}
		Hence, seeing that $p_{n_k}\leq p$ in $D_{n_k}$ and using   monotonicities of functions $\psi,\vphi$, we get that for all $k\geq 1$, 
		\begin{equation*}
			u\geq p_{n_k} \vphi\left(\frac{G_{D_{n_k}}(\xi\psi(p))}{p_{n_k}}\right)\quad\mbox{in }D_{n_k}. 
		\end{equation*}
		Consequently, by letting  $k$ tend to $\infty$ we obtain estimate~\eqref{estg} and the proof is complete. 
	\end{proof}
	
	In order to illustrate our results regarding  estimates of type~\eqref{estg}, we  give the following examples which deal with three different kinds of functions $\psi$.
	
	\medskip	
	\noindent{\it Example 1- Polynomial functions: }
	For $\psi(t)=t^\gamma$  we have   $\ell= 1/(1-\gamma)$ if $0<\gamma <1$ and $\ell=\infty$ if $\gamma \geq 1$. By elementary computation of $\varphi(t)$, Theorem~\ref{thecp} above yields that estimates given by~\eqref{estgr}  hold true for every   nonnegative solution to inequality~\eqref{sup-sol-g}. 	As  mentioned above, these estimates are already obtained by Grigor'yan and Verbitsky~\cite{grigor2019pointwise} provided that the solution  $u$ is  twice continuously  differentiable  in~$D$. Notice that, in this work, we deal  with continuous  solutions in the distributional sense. More generally, for  $\psi(t)=at+bt^\gamma$  where $a,b,\gamma>0$, we obtain
	$$
	\ell=\left\{\begin{array}{cl}
		\displaystyle\frac{a+b}{a(1-\gamma)}\log\frac{a+b}{b}& \mbox{if }\gamma<1,
		\\	\vspace{0,3cm}
		\infty & \mbox{if }\gamma>1,
	\end{array}\right.
	$$
	and by simple computation of  $\varphi(t)$, we deduce that  every nonnegative solution $u$ to
	\begin{equation}
		-Lu+\xi(au+bu^\gamma)\geq g\quad \mbox{in }D
	\end{equation}	
	satisfies, for all $x\in D$,
	\begin{equation*}
		u(x)\geq \left\{\begin{array}{ll}
			\displaystyle{p(x)\left(\frac{a+b}{a}\exp\Big(\frac{a}{a+b}(\gamma-1)\frac{G_D\xi(ap+bp^{\gamma})(x)}{p(x)} \Big)-\frac{b}{a}\right)_+^{1/(1-\gamma)}} & \mbox{if }\gamma<1,
			\vspace{0,3cm}
			\\ \displaystyle{ p(x)\left(\frac{a+b}{a}\exp\Big((\gamma-1)\frac{G_D\xi(ap+bp^{\gamma})(x)}{p(x)} \Big)-\frac{b}{a}\right)^{1/(1-\gamma)}} & \mbox{if }\gamma>1.
		\end{array}\right.
	\end{equation*}

	\medskip	
	\noindent{\it Example 2- Exponential functions:} For $\psi(t)=\sinh t=(e^t-e^{-t})/2 $, we prove that  $l=\infty$ and elementary computation  yields  that for every $t\geq 0$,
	$$
	\vphi(t)=2\,\mbox{artanh}\left(\alpha e^{-t}\right)\quad \mbox{with}\quad\alpha=\tanh\frac12.
	$$
	Hence, we conclude that for	every nonnegative solution $u$ to inequality
	\begin{equation}
		-Lu+\xi \sinh u\geq g\quad \mbox{in }D
	\end{equation}
	and for all $x\in D$, we have
	$$
	u(x)\geq 2\, p(x)\,\mbox{artanh}\left(\alpha\exp\Big({-\frac{G_D(\xi\sinh(p))(x)}{p(x)}}\Big)\right).
	$$

	\medskip	
	\noindent{\it Example 3- Logarithmic functions:}
	Consider $\psi(t)=a(1+bt)\log(1+bt)$ with $a,b>0$. By simple computation, we get that $\ell=\infty$ and for every real $t\geq 0$,
	$$
	\varphi(t)=\frac{(1+b)^{e^{-t}}-1}{b}.
	$$
	In particular, for $a=b=1$, it follows that every nonnegative solution $u$ to inequality 
	\begin{equation}
		-Lu+ \xi (1+u)\log(1+u)\geq g\quad \mbox{in }D,
	\end{equation}
	satisfies, for all $x\in D$,
	$$
	u(x)\geq p(x)\left(
	2^{\displaystyle{\exp\left( -\frac{G_D(\xi(1+p)\log(1+p))(x)}{p(x)}
			\right)}}
	-1\right)
	$$

\end{document}